\documentclass[twoside,leqno,10pt, A4]{amsart}
\usepackage{amsfonts}
\usepackage{amsmath}
\usepackage{amscd}
\usepackage{amssymb}
\usepackage{amsthm}
\usepackage{amsrefs}
\usepackage{latexsym}
\usepackage{mathrsfs}
\usepackage{bbm}
\usepackage{amscd}
\usepackage{amssymb}
\usepackage{amsthm}
\usepackage{amsrefs}
\usepackage{latexsym}
\usepackage{mathrsfs}
\usepackage{bbm}
\usepackage{enumerate}
\usepackage{graphicx}
\usepackage{color}
\begin{document}

\newtheorem{theorem}[subsection]{Theorem}
\newtheorem{proposition}[subsection]{Proposition}
\newtheorem{lemma}[subsection]{Lemma}
\newtheorem{corollary}[subsection]{Corollary}
\newtheorem{conjecture}[subsection]{Conjecture}
\newtheorem{prop}[subsection]{Proposition}
\newtheorem{defin}[subsection]{Definition}

\numberwithin{equation}{section}
\newcommand{\mr}{\ensuremath{\mathbb R}}
\newcommand{\mc}{\ensuremath{\mathbb C}}
\newcommand{\dif}{\mathrm{d}}
\newcommand{\intz}{\mathbb{Z}}
\newcommand{\ratq}{\mathbb{Q}}
\newcommand{\natn}{\mathbb{N}}
\newcommand{\comc}{\mathbb{C}}
\newcommand{\rear}{\mathbb{R}}
\newcommand{\prip}{\mathbb{P}}
\newcommand{\uph}{\mathbb{H}}
\newcommand{\fief}{\mathbb{F}}
\newcommand{\majorarc}{\mathfrak{M}}
\newcommand{\minorarc}{\mathfrak{m}}
\newcommand{\sings}{\mathfrak{S}}
\newcommand{\fA}{\ensuremath{\mathfrak A}}
\newcommand{\mn}{\ensuremath{\mathbb N}}
\newcommand{\mq}{\ensuremath{\mathbb Q}}
\newcommand{\half}{\tfrac{1}{2}}
\newcommand{\f}{f\times \chi}
\newcommand{\summ}{\mathop{{\sum}^{\star}}}
\newcommand{\chiq}{\chi \bmod q}
\newcommand{\chidb}{\chi \bmod db}
\newcommand{\chid}{\chi \bmod d}
\newcommand{\sym}{\text{sym}^2}
\newcommand{\hhalf}{\tfrac{1}{2}}
\newcommand{\sumstar}{\sideset{}{^*}\sum}
\newcommand{\sumprime}{\sideset{}{'}\sum}
\newcommand{\sumprimeprime}{\sideset{}{''}\sum}
\newcommand{\sumflat}{\sideset{}{^\flat}\sum}
\newcommand{\shortmod}{\ensuremath{\negthickspace \negthickspace \negthickspace \pmod}}
\newcommand{\V}{V\left(\frac{nm}{q^2}\right)}
\newcommand{\sumi}{\mathop{{\sum}^{\dagger}}}
\newcommand{\mz}{\ensuremath{\mathbb Z}}
\newcommand{\leg}[2]{\left(\frac{#1}{#2}\right)}
\newcommand{\muK}{\mu_{\omega}}
\newcommand{\thalf}{\tfrac12}
\newcommand{\lp}{\left(}
\newcommand{\rp}{\right)}
\newcommand{\Lam}{\Lambda_{[i]}}
\newcommand{\lam}{\lambda}
\newcommand{\af}{\mathfrak{a}}
\newcommand{\sw}{S_{[i]}(X,Y;\Phi,\Psi)}
\newcommand{\lz}{\left(}
\newcommand{\pz}{\right)}
\newcommand{\bfrac}[2]{\lz\frac{#1}{#2}\pz}
\newcommand{\odd}{\mathrm{\ primary}}
\newcommand{\even}{\text{ even}}
\newcommand{\res}{\mathrm{Res}}
\newcommand{\sumn}{\sumstar_{(c,1+i)=1}  w\left( \frac {N(c)}X \right)}
\newcommand{\lab}{\left|}
\newcommand{\rab}{\right|}
\newcommand{\Go}{\Gamma_{o}}
\newcommand{\Ge}{\Gamma_{e}}
\newcommand{\M}{\widehat}

\theoremstyle{plain}
\newtheorem{conj}{Conjecture}
\newtheorem{remark}[subsection]{Remark}

\makeatletter
\def\widebreve{\mathpalette\wide@breve}
\def\wide@breve#1#2{\sbox\z@{$#1#2$}%
     \mathop{\vbox{\m@th\ialign{##\crcr
\kern0.08em\brevefill#1{0.8\wd\z@}\crcr\noalign{\nointerlineskip}%
                    $\hss#1#2\hss$\crcr}}}\limits}
\def\brevefill#1#2{$\m@th\sbox\tw@{$#1($}%
  \hss\resizebox{#2}{\wd\tw@}{\rotatebox[origin=c]{90}{\upshape(}}\hss$}
\makeatletter

\title[Twisted first moment of central values of primitive quadratic Dirichlet $L$-functions]{Twisted first moment of central values of primitive quadratic Dirichlet $L$-functions}

%%\date{\today}
\author[P. Gao]{Peng Gao}
\address{School of Mathematical Sciences, Beihang University, Beijing 100191, China}
\email{penggao@buaa.edu.cn}

\author[L. Zhao]{Liangyi Zhao}
\address{School of Mathematics and Statistics, University of New South Wales, Sydney NSW 2052, Australia}
\email{l.zhao@unsw.edu.au}

\begin{abstract}
We evaluate the twisted first moment of central values of the family of primitive quadratic Dirichlet $L$-functions using the method of double Dirichlet series together with a recursive argument.  Our main result is an asymptotic formula with an error term of size that is the square root of that of the main term.
\end{abstract}

\maketitle



\section{Introduction}\label{sec 1}

The moments of central values of families of $L$-functions form important subjects of study in number theory as those results lead to a multitude of arithmetic consequences. Among the many families of $L$-functions, the quadratic ones are most thoroughly studied.   Notably, the family of primitive quadratic Dirichlet $L$-functions has been first investigated by M. Jutila \cite{Jutila}, who obtained asymptotic formulas for the first and second moment. The error term for the first moment in Jutila's work were improved by D. Goldfeld and J. Hoffstein \cite{DoHo}, A. I. Vinogradov and L. A. Takhtadzhyan  \cite{ViTa} and M. P. Young \cite{Young1}. \newline

   Currently, the best known error term for the smoothed first moment of central values of primitive quadratic Dirichlet $L$-functions were obtained in \cite{DoHo} and \cite{Young1}, which employed different approaches, the powerful method of double Dirichlet series adapted in \cite{DoHo} and a recursive argument used in \cite{Young1}. \newline

  The aim of the present paper is to give an alternative proof of the above mentioned result in \cites{DoHo, Young1}, by merging the two above-mentioned methods of double Dirichlet series and recursion. The starting point is to apply the recursive argument to connect the first moment of primitive quadratic Dirichlet $L$-functions to that of quadratic Dirichlet $L$-functions over all odd moduli. This process naturally leads to the consideration on twisted first moments of the families of $L$-functions under our consideration. We therefore need to enlarge our investigation by studying these twisted moments instead. We then fulfil the task by evaluating the twisted first moment of quadratic Dirichlet $L$-functions over all conductors asymptotically using the method of double Dirichlet series by extending our arguments in \cite{G&Zhao20} for the untwisted case. \newline

  To precisely state our results, let $\chi^{(m)}=\leg {m}{\cdot}$ denote the Kronecker symbol defined for any integer $m \equiv 0$, $1 \pmod 4$, following the notation given on \cite[p. 52]{iwakow}. We also write $\chi_{m}=\leg {\cdot}{m}$ for the Jacobi symbol with an odd integer $m$.  Given any $L$-function, $L^{(c)}$ stands for the function given by the Euler product defining $L$ but omitting those primes dividing $c$. As usual, $\zeta(s)$ is the Riemann zeta function and the letter $p$ denotes a prime throughout the paper. Let $l>0$ be an odd, square-free integer and $w(t)$ be a fixed non-negative smooth function compactly supported on ${\mr}^+$, the set of positive real numbers.  We are primarily interested in the twisted first moment of the central values of the family of primitive quadratic Dirichlet $L$-functions given by
\begin{align}
\label{ZetWithCharacters}
  \sumstar_{(d,2)=1}L(\tfrac 12, \chi^{(8d)})\chi^{(8d)}(l)w \bfrac {d}X,
\end{align}
  where $\sum^*$ indicates that the sum runs over square-free integers.  The above family of $L$-functions was first considered by K. Soundararajan in \cite{sound1}. Here we note that (see \cite{sound1}) that if $d$ is positive, odd and square-free, the character $\chi^{(8d)}$ is primitive modulo $8d$ such that $\chi^{(8d)}(-1)=1$. \newline

    As alluded to above, our understanding of the expression in \eqref{ZetWithCharacters} requires the evaluation of the twisted first moment of central values of the family of quadratic Dirichlet $L$-functions over all conductors given by
\begin{align*}
%%\label{ZetWithallCharacters}
  \sum_{(d,2)=1}L(\tfrac 12, \chi^{(8d)})\chi^{(8d)}(l)w \bfrac {d}X.
\end{align*}

   In fact, we shall study a shifted moment in which $\tfrac12$ in \eqref{ZetWithCharacters} is replaced by $\tfrac12+\alpha$ with $0 < |\Re(\alpha)| < \tfrac12$.  Moreover, we apply the M\"obius function to remove the condition that $d$ is square-free in \eqref{ZetWithCharacters}.  This gives
\begin{equation} \label{squarefreetoall}
	\sumstar_{\substack{(d,2)=1 }}L(\tfrac{1}{2}+\alpha, \chi^{(8d)})\chi^{(8d)}(l)w \bfrac {d}X=\sum_{(a, 2l)=1}\mu(a)\sum_{\substack{(d,2)=1 }}L(\tfrac{1}{2}+\alpha, \chi^{(8da^2)})\chi^{(8d)}(l)w \bfrac {da^2}X.
\end{equation}
  We then separate the sum over $a$ of the right-hand side expression above by considering the terms with $a \leq Y$ and with $a > Y$ for a parameter $Y$ to be specified later, so that we recast the right-hand side \eqref{squarefreetoall} as $M_1+M_2$ with
\begin{align}
\label{M1M2}
\begin{split}
 M_1:=& \sum_{\substack{a \leq Y \\ (a, 2l)=1}}\mu(a)\sum_{\substack{(d,2)=1 }}L(\tfrac{1}{2}+\alpha, \chi^{(8da^2)})\chi^{(8d)}(l)w \bfrac {da^2}X \quad \mbox{and} \\
 M_2:=& \sum_{\substack{a > Y \\ (a, 2l)=1}}\mu(a)\sum_{\substack{(d,2)=1 }}L(\tfrac{1}{2}+\alpha, \chi^{(8da^2)})\chi^{(8d)}(l)w \bfrac {da^2}X.
\end{split}
\end{align}

Observe that
\begin{align}
\label{Lrel}
 L(\tfrac{1}{2}+\alpha, \chi^{(8da^2)})=L(\tfrac{1}{2}+\alpha, \chi^{(8d)})\prod_{p|a} \left( 1-\frac {\chi^{(8d)}(p)}{p^{1/2+\alpha}} \right)=L(\tfrac{1}{2}+\alpha, \chi^{(8d)})\sum_{r | a}\frac {\mu(r)\chi^{(8d)}(r) }{r^{1/2+\alpha}}.
\end{align}
The identity \eqref{Lrel} thus transforms $M_1$ as
\begin{align}
\label{M1simplified}
 M_1=& \sum_{\substack{a \leq Y \\ (a, 2l)=1}}\mu(a)\sum_{r | a}\frac {\mu(r)}{r^{1/2+\alpha}}\sum_{\substack{(d,2)=1 }}L(\tfrac{1}{2}+\alpha, \chi^{(8d)})\chi^{(8d)}(rl)w \bfrac {da^2}X.
\end{align}

On the other hand, $M_2$, via a change of variables $d \rightarrow db^2$, becomes
\begin{align*}
%%\label{M2simplified}
 M_2 =& \sumstar_{(d,2)=1}\sum_{(b,2l)=1}\sum_{\substack{a > Y \\ (a, 2l)=1}}\mu(a)L(\tfrac{1}{2}+\alpha, \chi^{(8d(ab)^2)})\chi^{(8d)}(l)w \bfrac {d(ab)^2}X.
\end{align*}
Setting $c=ab$ and applying \eqref{Lrel} lead to
\begin{align}
\label{M2simplifiedc}
\begin{split}
 M_2 =& \sumstar_{(d,2)=1}\sum_{(c,2l)=1}\Big(\sum_{\substack{a|c \\ a > Y }}\mu(a)\Big)L(\tfrac{1}{2}+\alpha, \chi^{(8dc^2)})\chi^{(8d)}(l)w \bfrac {dc^2}X \\
 =& \sum_{(c,2l)=1}\Big(\sum_{\substack{a|c \\ a > Y }}\mu(a)\Big)\sum_{r | c}\frac {\mu(r)}{r^{1/2+\alpha}}\sumstar_{(d,2)=1}L(\tfrac{1}{2}+\alpha, \chi^{(8d)})\chi^{(8d)}(lr)w \bfrac {dc^2}X.
\end{split}
\end{align}

  We shall evaluate $M_1$ and $M_2$ using different strategies. For $M_1$, we shall apply the method of double Dirichlet series to treat the inner-most sum on the right-hand side of \eqref{M1simplified} both unconditionally and under the Riemann hypothesis (RH).  To deal with $M_2$, we shall employ a recursive argument due to Young \cite{Young1}. The formula we need for $M_1$ is as follows.
\begin{theorem}
\label{Theorem for all characters}
 With the notation as above, let $l>0$ be any odd, square-free integer and $w(t)$ a non-negative smooth function compactly supported on ${\mr}^+$ with $\widehat w(s)$ being its Mellin transform.  For $0<|\Re(\alpha)|<1/2$ and any $\varepsilon>0$, we have
\begin{align}
\label{Asymptotic for ratios of all characters}
\begin{split}	
 \sum_{\substack{(d,2)=1}}& L(\tfrac{1}{2}+\alpha, \chi^{(8d)})\chi^{(8d)}(l) w \bfrac {d}X \\
=&  X\M w(1)\frac {\zeta^{(2)}(1+2\alpha)}{2l^{1/2+\alpha}\zeta^{(2)}(2+2\alpha)}\prod_{p|l}\frac{1-p^{-1}}{1-p^{-2-2\alpha}} +X^{1-\alpha}\M w(1-\alpha) \gamma_{\alpha}l^{-1/2+\alpha} \frac {\zeta^{(2)}(1-2\alpha)}{2\zeta^{(2)}(2)} \prod_{p|l} \Big(1+p^{-1} \Big)^{-1} \\
& \hspace*{3cm} +O\lz(1+|\alpha|)^{2+\varepsilon}l^{1/2+\varepsilon}X^{\delta_{\alpha}+\varepsilon}\pz .
\end{split}
\end{align}
Here we can take $\delta_{\alpha}=1/2-\Re(\alpha)$ unconditionally and $\delta_{\alpha}=1/4-\Re(\alpha)$ assuming the truth of RH.  Moreover,
\begin{align}
\label{gammadef}
\begin{split}
\gamma_{\alpha}=\leg {8}{\pi}^{-\alpha}\frac {\Gamma(\tfrac{1}{4}-\tfrac{\alpha}{2})}{\Gamma(\tfrac{1}{4}+\tfrac{\alpha}{2})}.
\end{split}
\end{align}
\end{theorem}

    Theorem \ref{Theorem for all characters} now allows us to evaluate $M_1$ asymptotically. Moreover, $M_1$ becomes the expression given in \eqref{ZetWithCharacters} upon taking $Y$ to $\infty$ in \eqref{M1M2}.  In fact, taking $Y \to \infty$ and summing trivially the error term using  $\delta_{\alpha}=1/2-\Re(\alpha)$ in \eqref{Asymptotic for ratios of all characters} establish the following result on the twisted first moment of central values of the family of primitive quadratic Dirichlet $L$-functions.
\begin{theorem}
\label{Theorem for primitive characters}
 With the notation as above and for any $\varepsilon>0$, we have uniformly for $\Re(\alpha) \ll (\log{X})^{-1}$ and $\Im(\alpha) \ll X^{\varepsilon}$,
\begin{align}
\label{Asymptotic for ratios of primitive characters}
\begin{split}	
 \sumstar_{\substack{(d,2)=1}} & L(\tfrac{1}{2}+\alpha, \chi^{(8d)})\chi^{(8d)}(l) w \bfrac {d}X \\
=& \frac{X \widehat{w}(1)}{2 \zeta^{(2)}(2)} l^{-1/2 - \alpha} \zeta^{(2)}(1 + 2\alpha) B_{\alpha}(l)
+ \frac{X^{1-\alpha} \widehat{w}(1-\alpha) \gamma_{\alpha}}{2 \zeta^{(2)}(2)} l^{-1/2 + \alpha} \zeta^{(2)}(1 - 2\alpha) B_{-\alpha}(l)  + O((lX)^{1/2 + \varepsilon}),
\end{split}
\end{align}
  where $B_{\alpha}$ is defined by
\begin{equation*}
\zeta^{(2)}(1 + 2\alpha) B_{\alpha}(l) = \sum_{(n,2)=1} \frac{1}{n^{1 + 2\alpha}} \prod_{p | nl} (1 + p^{-1})^{-1}.
\end{equation*}
\end{theorem}

  The above result has been established by M. P. Young in \cite{Young1} under the slightly more restrictive conditions $\Re(\alpha) \ll (\log{X})^{-1}$ and $\Im(\alpha)  \ll (\log{X})^2$.  Here we point out that although one does not really need to apply any recursive argument to establish Theorem \ref{Theorem for primitive characters}, we find it useful to demonstrate the impact of various error terms in \eqref{Asymptotic for ratios of all characters} on the result given in \eqref{Asymptotic for ratios of primitive characters}. Moreover, this type of argument may have potential applications to other problems. For this reason, we give a proof of Theorem \ref{Theorem for primitive characters} via a recursive argument to establish the following result.
\begin{theorem}
\label{theo:recursive}
 With the notation as above, suppose that \eqref{Asymptotic for ratios of all characters} holds but with an error term of size $l^{1/2 + \varepsilon} X^{\delta + \varepsilon}$ for some $0< \delta \leq 1/2$ and any $\Re(\alpha) \ll (\log{X})^{-1}$, $\Im(\alpha) \ll X^{\varepsilon}$.  Further assume that \eqref{Asymptotic for ratios of primitive characters} holds but with an $O$-term of size $l^{1/2 + \varepsilon} X^{f+ \varepsilon}$ for some $1/2 \leq f \leq 1$ and any $\Re(\alpha) \ll (\log{X})^{-1}$, $\Im(\alpha) \ll X^{\varepsilon}$.  Then \eqref{Asymptotic for ratios of primitive characters}  holds with an $O$-term of size $l^{1/2 + \varepsilon} X^{1/2 + \varepsilon}$.
\end{theorem}

   It follows from Theorem \ref{Theorem for primitive characters} that one can take $\delta=1/2$ in Theorem \ref{theo:recursive}. As pointed out on \cite[p. 81]{Young1}, one can also take $f=1$ in Theorem \ref{theo:recursive} due to a second moment result of Jutila \cite{Jutila}. Thus we deduce from Theorem \eqref{theo:recursive} that Theorem \ref{Theorem for primitive characters} is valid. We note here that Theorem \ref{theo:recursive} illustrates an interesting phenomenon that an improvement on the error term for the power of $X$ in \eqref{Asymptotic for ratios of all characters} does not necessarily lead to an improvement on the error term in \eqref{Asymptotic for ratios of primitive characters}. \newline

  As the error term in \eqref{Asymptotic for ratios of all characters} is uniform in $\alpha$, we set $l=1$ and take the limit as $\alpha \rightarrow 0$ to recovers the main result in \cite{Young1} that evaluates asymptotically the first moment of central values of the family of primitive quadratic Dirichlet $L$-functions.
\begin{corollary}
\label{Thmfirstmomentatcentral}
	With the notation as above, we have,
\begin{align*}
%%\label{Asymptotic for first moment at central}
\begin{split}			
	& 	\sumstar_{\substack{(d,2)=1}}L(\tfrac{1}{2}, \chi^{(8d)}) w \bfrac {d}X  = XQ(\log X)+O\lz  X^{1/2+\varepsilon}\pz.
\end{split}
\end{align*}
  where $Q$ is a linear polynomial whose coefficients depend only on $\M w(1)$ and $\M w'(1)$.
\end{corollary}

  We omit the explicit expression of $Q$ here since our main focus is the error term.

\section{Preliminaries}
\label{sec 2}

\subsection{Functional equations for Dirichlet $L$-functions}
\label{sec 2.1}

  We recall that $\chi_n=\left(\frac {\cdot}{n} \right)$ for any odd positive integer $n$. We also write $\psi_j=\chi^{(4j)}$ for $j = \pm 1, \pm 2$ where we recall that $\chi^{(d)}=\leg {d}{\cdot} $ is the Kronecker symbol for integers $d \equiv 0, 1 \pmod 4$. Note that each $\psi_j$ is a character modulo $4|j|$.  Let $\psi_0$ stand for the primitive principal character. \newline

  For any Dirichlet character $\chi$ modulo $n$ and any integer $q$, recall that the usual Gauss sum $\tau(\chi,q)$ is defined as
\begin{equation*}
%%\label{key}
		\tau(\chi,q)=\sum_{j\shortmod n}\chi(j)e \left( \frac {jq}n \right), \quad \mbox{where} \quad  e(z)= \exp(2 \pi i z).
\end{equation*}
	
  We write
\begin{equation*}
%%\label{key}
		\Ge(s)=\frac{\Gamma\bfrac{1-s}{2}}{\Gamma\bfrac s2} \quad \mbox{and} \quad
		\Go(s)=\frac{\Gamma\bfrac{2-s}{2}}{\Gamma\bfrac {s+1}2}.
\end{equation*}
   We further define $\Gamma_{e/o}(s;\chi)$ for any Dirichlet character $\chi$ so that
\begin{equation*}
%%\label{key}
			\Gamma_{e/o}(s;\chi)=\begin{cases}
				\Gamma_{e}(s),&\hbox{if $\chi(-1)=1$},\\
				\Gamma_{o}(s),&\hbox{if $\chi(-1)=-1$.}
			\end{cases}
\end{equation*}

	We quote the following functional equation from \cite[Proposition 2.3]{Cech1} for all Dirichlet characters $\chi$ modulo $n$, which plays a crucial role in the proof of Theorem \ref{Theorem for all characters}.
\begin{lemma}
\label{Functional equation with Gauss sums}
		With the notation as above, we have for any non-principal Dirichlet character $\chi$,
\begin{equation*}
%%\label{Equation functional equation with Gauss sums}
			L(s,\chi)=\epsilon(\chi)\frac{\pi^{s-1/2}}{n^s}\Gamma_{e/o}(s;\chi)  K(1-s,\chi),
\end{equation*}
where
\begin{equation*}
%%\label{key}
		K(s,\chi)=\sum_{q=1}^\infty\frac{\tau(\chi,q)}{q^s} \quad \mbox{and} \quad	\epsilon(\chi)=\begin{cases}
				1,& \hbox{if $\chi(-1)=1$},\\
				-i,&\hbox{if $\chi(-1)=-1$.}
			\end{cases}
\end{equation*}
\end{lemma}

     Note that if $\chi$ is a primitive character modulo $n$, then (see \cite[\S 9, (2)]{Da})
\begin{equation*}
%%\label{Gauss sums for primitive characters}
		\tau(\chi,q)=\bar\chi(q)\tau(\chi,1).
\end{equation*}
Therefore, similar to a Dirichlet $L$-functions, $K(s,\chi)$ converges absolutely for $\Re (s) > \tfrac12$ and admits analytic continuation to $\comc$.   Further, it is well-known (see \cite[\S 2]{Da})  that $\epsilon(\chi)\tau(\chi,1)=n^{1/2}$ in this case, so that Lemma \ref{Functional equation with Gauss sums}
implies that for such $\chi$,
\begin{equation}
\label{fcneqnprimitive}
			L(s,\chi)=\frac{\pi^{s-1/2}}{n^{s-1/2}}\Gamma_{e/o}(s;\chi)L(1-s,\bar\chi).
\end{equation}

    We shall apply the above lemma to the case when $\chi=\chi^{(4)}\chi_n$ for any odd positive integer $n$ not a perfect square (henceforth written as $n \neq \square$). Before proceeding any further, we first note that  for two Dirichlet characters $\chi_i$ modulo $n_i, i=1,2$ with $(n_1, n_2)=1$, then for $\chi_1\chi_2$ considered as a Dirichlet character modulo $n_1n_2$, the Chinese Remainder Theorem gives that (see \cite[Lemma 2.1]{Cech1})
\begin{equation*}
%%\label{key}
		\tau(\chi_1 \chi_2,q)=\chi_1(n_2)\chi_2(n_1)\tau(\chi_1,q)\tau(\chi_2,q).
\end{equation*}

    Moreover, direct calculation renders
\begin{equation*}
%%\label{key}
		\tau\lz\chi^{(4)},q\pz=\begin{cases}
					0,&\hbox{if $q$ is odd,}\\
					-2,&\hbox{if $q\equiv2\mod 4$,}\\
					2,&\hbox{if $q\equiv0\mod 4$.}
				\end{cases}
\end{equation*}

   It follows from the above that for any odd positive integer $n \neq \square$,
\begin{equation}
\label{tauvalue}
		\tau(\chi^{(4)}\chi_n,q)=\tau(\chi^{(4)},q)\tau(\chi_n,q)=\begin{cases}
					0,&\hbox{if $q$ is odd,}\\
					-2\tau(\chi_n,q),&\hbox{if $q\equiv2\pmod 4$,}\\
					2\tau(\chi_n,q),&\hbox{if $q\equiv0\pmod 4$.}
				\end{cases}
\end{equation}

   For the purpose of this paper, we also define an associated Gauss sum $G\lz\chi_n,q\pz$ by
\begin{align}
\label{Gdef}
\begin{split}
			G\lz\chi_n,q\pz&=\lz\frac{1-i}{2}+\leg{-1}{n}\frac{1+i}{2}\pz\tau\lz\chi_n,q\pz=\begin{cases}
				\tau\lz\chi_n,q\pz,&\hbox{if $n\equiv1 \pmod 4$,}\\
				-i\tau\lz\chi_n,q\pz,&\hbox{if $n\equiv3\pmod 4$}.
			\end{cases}
\end{split}
\end{align}

  The advantage of $G\lz\chi_n,q\pz$ over $\tau\lz\chi_n,q\pz$ is that $G\lz\chi_n,q\pz$ is now a multiplicative function of $n$. In fact, writing $\varphi(m)$ for the Euler totient function of $m$, emerges the following evaluation of $G\lz\chi_n,q\pz$ from \cite[Lemma 2.3]{sound1}.
\begin{lemma}
\label{lem:Gauss}
   If $(m,n)=1$, then $G(\chi_{mn},q)=G(\chi_m,q)G(\chi_n,q)$. Suppose that $p^a$ is
   the largest power of $p$ dividing $q$ (put $a=\infty$ if $q=0$).
   Then for $k \geq 0$,
\begin{equation*}
%%\label{Sound's Gauss sums - exact formula}
		G\lz\chi_{p^k},q\pz=\begin{cases}\varphi(p^k),&\hbox{if $k\leq a$, $k$ even,}\\
			0,&\hbox{if $k\leq a$, $k$ odd,}\\
			-p^a,&\hbox{if $k=a+1$, $k$ even,}\\
			\leg{qp^{-a}}{p}p^{a}\sqrt p,&\hbox{if $k=a+1$, $k$ odd,}\\
			0,&\hbox{if $k\geq a+2$}.
		\end{cases}
\end{equation*}
\end{lemma}
	
   We deduce from Lemma \ref{Functional equation with Gauss sums}, \eqref{tauvalue} and \eqref{Gdef} the following functional equation concerning $L(s,\chi^{(4)}\chi_n)$ for any odd $n \neq \square$, treating $\chi^{(4)}\chi_n$ as a Dirichlet character modulo $4n$.
\begin{corollary}
\label{Functional equation for quadLfcn}
		With the notations as above, we have, for any odd $n \neq \square$,
\begin{equation*}
%%\label{Equation functional equation quadLfcn}
			L(s,\chi^{(4)}\chi_n)=-\frac{2\pi^{s-1/2}}{(4n)^s}\Gamma_{e/o}(s;\chi_n) \sum_{\substack{q\equiv2\pmod 4}}^\infty\frac{G(\chi_n,q)}{q^{1-s}}+\frac{2\pi^{s-1/2}}{(4n)^s}\Gamma_{e/o}(s;\chi_n) \sum_{\substack{q\equiv 0 \pmod 4}}^\infty\frac{G(\chi_n,q)}{q^{1-s}}.
\end{equation*}
The sums on the right-hand side of the above converge absolutely for $\Re (s) < 0$.
\end{corollary}

\subsection{Bounding $L$-functions}

  We cite the following large sieve result bounding $|L( s,  \chi)|$ on average for $\Re(s) \geq 1/2$ from \cite[Lemma 2.7]{G&Zhao20}, which is a consequence of \cite[Theorem 2]{DRHB}.
\begin{lemma} \label{lem:2.3}
 With the notation as above, let $S(X)$ denote the set of real characters $\chi$ with conductor not exceeding $X$. Then, for any $s \in \comc$ and $\varepsilon_0$, $\varepsilon>0$ with $\Re(s) \geq 1/2$, $|s-1|>\varepsilon_0$,
\begin{align}
\label{L1estimation}
\sum_{\substack{\chi \in S(X)}} |L(s, \chi)|
\ll & X^{1+\varepsilon} |s|^{1/4+\varepsilon}.
\end{align}
\end{lemma}

\subsection{Some results on multivariable complex functions} Here, we quote some results necessary for our proofs from multivariable complex analysis.  First, we need the notation of a tube domain.
\begin{defin}
		An open set $T\subset\mc^n$ is a tube if there is an open set $U\subset\mr^n$ such that $T=\{z\in\mc^n:\ \Re(z)\in U\}.$
\end{defin}
	
   For a set $U\subset\mr^n$, we define $T(U)=U+i\mr^n\subset \mc^n$.  We shall utilize the following Bochner's Tube Theorem \cite{Boc}.
\begin{theorem}
\label{Bochner}
		Let $U\subset\mr^n$ be a connected open set and $f(z)$ a function holomorphic on $T(U)$. Then $f(z)$ has a holomorphic continuation to the convex hull of $T(U)$.
\end{theorem}

 Let $\widehat T$ denote the convex hull of an open set $T\subset\mc^n$.  This next result \cite[Proposition C.5]{Cech1} gives bounds on the moduli of holomorphic continuations of multivariable complex functions.  In brief, it asserts that holomorphic continuations inherit the bounds on the functions from which they emanate.
\begin{prop}
\label{Extending inequalities}
		Assume that $T\subset \mc^n$ is a tube domain, $g,h:T\rightarrow \mc$ are holomorphic functions, and let $\tilde g,\tilde h$ be their holomorphic continuations to $\widehat T$. If  $|g(z)|\leq |h(z)|$ for all $z\in T$ and $h(z)$ is nonzero in $T$, then also $|\tilde g(z)|\leq |\tilde h(z)|$ for all $z\in \widehat T$.
\end{prop}

\section{Proof of Theorem \ref{Theorem for all characters}}

  For $\Re(s)$, $\Re(w)$ sufficiently large, set
\begin{align}
\label{Aswfexp}
\begin{split}
A(s,w; l)=& \sum_{\substack{(d,2)=1 }}\frac{L(w, \chi^{(8d)})\chi^{(8d)}(l)}{d^s}
=\sum_{\substack{(dm,2)=1}}\frac{\chi^{(8d)}(ml)}{m^wd^s}= \sum_{\substack{(m,2)=1}}\frac{\chi^{(8)}(ml)L^{(2)}( s,\chi_{ml} )}{m^w}.
\end{split}
\end{align}

  In the next two sections, we develop some analytic properties of $A(s,w;l)$, necessary in establishing Theorem \ref{Theorem for all characters}.
	
\subsection{First region of absolute convergence of $A(s,w;l)$}

   We use the first equality in \eqref{Aswfexp} to arrive at, reducing the sum over $d$ to a sum over square-free integers,
\begin{align}
\label{Abound}
\begin{split}
		A(s,w;l)= & \sum_{\substack{(h,2l)=1}}\frac {1}{h^{2s}}\sumstar_{\substack{(d,2)=1}}\frac{L(w,  \chi^{(8d)})\chi^{(8d)}(l)\prod_{p | h}(1- \chi^{(8d)}(p)p^{-w}) }{d^s}
= \zeta^{(2l)}(2s)\sumstar_{\substack{(d,2)=1}}\frac{L(w,  \chi^{(8d)})\chi^{(8d)}(l)}{d^sL^{(2l)}(2s+w,  \chi^{(8d)})}.
\end{split}
\end{align}

  Now \eqref{L1estimation} and partial summation yield that the sum in the last expression of \eqref{Abound} is convergent for $\Re(s)>1$,  $\Re(w) \geq 1/2$ as well as for $\Re(2s)>1$, $\Re(2s+w)>1$, $\Re(s+w)>3/2$, $\Re(w) < 1/2$.  Hence, $A(s,w; l)$ converges absolutely in the region
\begin{equation*}
%%\label{key}
		S_0=\{(s,w): \Re(s)>1,\ \Re(2s+w)>1,\ \Re(s+w)>3/2\}.
\end{equation*}

As the condition $\Re(2s+w)>1$ is contained in the other conditions, the description of $S_0$ simplifies to
\begin{equation*}
%%\label{key}
		S_0=\{(s,w): \Re(s)>1, \ \Re(s+w)>3/2\}.
\end{equation*}

 Next, upon writing $m=m_0m^2_1$ with $m_0$ odd and square-free, the last expression in \eqref{Aswfexp} is recast as
\begin{align}
\label{Sum A(s,w,z) over n}	
\begin{split}	
A(s,w;l)=& \sum_{\substack{(m_1,2)=1}}\frac{1}{m_1^{2w}}\sumstar_{\substack{(m_0,2)=1}}\frac{\chi^{(8)}(m_0l)L^{(2)}( s, \chi_{m_0l})\prod_{p | m_1}(1-\chi_{m_0l}(p)p^{-s}) }{m_0^{w}} \\
=& \zeta^{(2)}(2w)\sumstar_{\substack{(m_0,2)=1}}\frac{\chi^{(8)}(m_0l)L^{(2)}( s, \chi_{m_0l})}{m_0^{w}L^{(2)}(2w+ s, \chi_{m_0l})}.
\end{split}
\end{align}

  Arguing as above by making use of \eqref{L1estimation} and partial summation again, the sum over $m_0$ in \eqref{Sum A(s,w,z) over n} converges absolutely in
\begin{align*}
%%\label{key}
		S_1=& \{(s,w): \Re(w)>1, \ \Re(s+w)>3/2\}.
\end{align*}

The convex hull of $S_0 \cup S_1$ is
\begin{equation*}
%%\label{Region of convergence of A(s,w,z)}
		S_2=\{(s,w):\ \Re(s+w)> 3/2 \}.
\end{equation*}

Hence, Theorem \ref{Bochner} implies that $(s-1)(w-1)A(s,w;l)$ can be holomorphically continued to the region $S_2$.

\subsection{Second region of absolute convergence of $A(s,w;l)$}

 From \eqref{Aswfexp}, we infer
\begin{align}
\begin{split}
\label{A1A2}
 A(s,w; l) =& \sum_{\substack{(m,2)=1 \\ ml =  \square}}\frac{\chi^{(8)}(ml)L^{(2)}( s,\chi_{ml} )}{m^w} +\sum_{\substack{(m,2)=1 \\ ml \neq \square}}\frac{\chi^{(8)}(ml)L( s,\chi^{(4)}\chi_{ml} )}{m^w} := \ A_1(s,w;l)+A_2(s,w;l).
\end{split}
\end{align}

   As $l$ is square-free, we deduce, in a manner similar to the derivation of \eqref{Sum A(s,w,z) over n}, that
\begin{align*}
\begin{split}
%%\label{A1def}
 A_1(s,w;l) = \sum_{\substack{(m,2)=1 \\ ml =  \square}}\frac{\chi^{(8)}(ml)L^{(2)}( s,\chi_{ml} )}{m^w}
= \frac{ \zeta^{(2)}(2w)L^{(2)}( s, \chi_{l^2})}{l^{w}L^{(2)}(2w+ s, \chi_{l^2})},
\end{split}
\end{align*}
  It follows from the above that save for a simple pole at $s=1$, $A_1(s,w; l)$ is holomorphic in
\begin{align}
\label{S3}
	S_3=\Big\{(s,w):\ &  \Re(s+2w)>1, \Re (w) > 3/4  \Big\}.
\end{align}

   For two Dirichlet characters $\psi$ and $\psi'$ with conductors dividing $8$, we define
\begin{align}
\begin{split}
\label{C12def}
	C_1(s,w;l,\psi,\psi'):= \sum_{\substack{m,q\geq 1 \\ (m,2)=1}}\frac{G\lz \chi_{ml},q\pz\psi(m)\psi'(q) }{m^wq^s} \quad \mbox{and} \quad C_2(s,w;l,\psi,\psi'):=  \sum_{\substack{m,q\geq 1 \\ (m,2)=1 \\ ml=\square}}\frac{G \lz \chi_{ml},q\pz\psi(m) \psi'(q) }{m^{w}q^s}.
\end{split}
\end{align}

Using the functional equation in Corollary~\ref{Functional equation for quadLfcn}, we arrive at
\begin{align}
\label{preA2exp}
\begin{split}
 A_2(s,w; l)
 =& -\frac{2\pi^{s-1/2}}{(4l)^s}\Gamma_{e}(s)\sum_{\substack{m, q \geq 1 \\ (m,2)=1 \\ ml \neq \square \\ ml \equiv 1 \pmod 4 \\ q \equiv 2 \pmod 4}}\frac{\chi^{(8)}(ml)G\lz \chi_{ml},q\pz}{m^{w+s}q^{1-s}}-\frac{2\pi^{s-1/2}}{(4l)^s}\Gamma_{o}(s)\sum_{\substack{m, q \geq 1 \\ (m,2)=1 \\ ml \neq \square \\ ml \equiv -1 \pmod 4 \\ q \equiv 2 \pmod 4}}\frac{\chi^{(8)}(ml)G\lz \chi_{ml},q\pz}{m^{w+s}q^{1-s}} \\
 & \hspace*{1cm}+\frac{2\pi^{s-1/2}}{(4l)^s}\Gamma_{e}(s)\sum_{\substack{m, q \geq 1 \\ (m,2)=1 \\ ml \neq \square \\ ml \equiv 1 \pmod 4 \\ q \equiv 0 \pmod 4}}\frac{\chi^{(8)}(ml)G\lz \chi_{ml},q\pz}{m^{w+s}q^{1-s}}+\frac{2\pi^{s-1/2}}{(4l)^s}\Gamma_{o}(s)\sum_{\substack{m, q \geq 1 \\ (m,2a)=1 \\ ml \neq \square \\ ml \equiv -1 \pmod 4\\ q \equiv 0 \pmod 4 }}\frac{\chi^{(8)}(ml)G\lz \chi_{ml},q\pz}{m^{w+s}q^{1-s}} .
 \end{split}
 \end{align}
With the identity $G(\chi_n, 4q) = G(\chi_n, q)$, the last two terms on the right-hand side of \eqref{preA2exp} can be rewritten as
\[ \frac{2^{2s-1}\pi^{s-1/2}}{(4l)^s}\Gamma_{e}(s)\sum_{\substack{m, q \geq 1 \\ (m,2)=1 \\ ml \neq \square \\ ml \equiv 1 \pmod 4 }}\frac{\chi^{(8)}(ml)G\lz \chi_{ml},q\pz}{m^{w+s}q^{1-s}}+\frac{2^{2s-1}\pi^{s-1/2}}{(4l)^s}\Gamma_{o}(s)\sum_{\substack{m, q \geq 1 \\ (m,2)=1 \\ ml \neq \square \\ ml \equiv -1 \pmod 4}}\frac{\chi^{(8)}(ml)G\lz \chi_{ml},q\pz}{m^{w+s}q^{1-s}} . \]
  Using the characters $\psi_i$ introduced at the beginning of Section \ref{sec 2.1} to detect the various congruence conditions on $ml$ and $q$ in \eqref{preA2exp}, we transform $A_2(s,w; l)$ in the following way.
\begin{align}
\label{A2exp}
\begin{split}
 A_2(s,w; l)
=& -\frac{2^s\pi^{s-1/2}}{(4l)^s}\frac {\Gamma_{e}(s)}{2}\sum^2_{i=1}(-1)^{i+1}\big (C_i(1-s,w+s;l,\psi_1, \psi_1)+C_i(1-s,w+s;l,\psi_{-1}, \psi_{1}) \big ) \\
&-\frac{2^s\pi^{s-1/2}}{(4l)^s}\frac {\Gamma_{o}(s)}{2}\sum^2_{i=1}(-1)^{i+1}\big (C_i(1-s,w+s;l,\psi_1, \psi_1)-C_i(1-s,w+s;l,\psi_{-1}, \psi_{1}) \big )  \\
&+\frac{2^{2s-1}\pi^{s-1/2}}{(4l)^s}\frac {\Gamma_{e}(s)}{2}\sum^2_{i=1}(-1)^{i+1}\big (\psi_2(l)C_i(1-s,w+s;l,\psi_2, \psi_0)+\psi_{-2}(l)C_i(1-s,w+s;l,\psi_{-2}, \psi_{0}) \big ) \\
&+\frac{2^{2s-1}\pi^{s-1/2}}{(4l)^s}\frac {\Gamma_{o}(s)}{2}\sum^2_{i=1}(-1)^{i+1}\big (\psi_2(l)C_i(1-s,w+s;l,\psi_2, \psi_0)-\psi_{-2}(l)C_i(1-s,w+s;l,\psi_{-2}, \psi_{0}) \big ).
\end{split}
\end{align}

  Following the approach by K. Soundararajan and M. P. Young in \cite[\S 3.3]{S&Y} to write every integer $q \geq 1$ uniquely as $q=q_1q^2_2$ with $q_1$ square-free, we get that
\begin{equation}
\label{Cidef}
		C_i(s,w;l, \psi,\psi')=\sumstar_{q_1}\frac{\psi'(q_1)}{q_1^s}\cdot D_i(s, w; l, q_1,\psi, \psi'), \quad i =1,\; 2,
\end{equation}
	where
\begin{align*}
%%\label{D12def}
\begin{split}
%%\label{key}
		D_1(s, w; l,q_1, \psi,\psi') := \sum_{\substack{m,q_2=1 \\ (m,2)=1}}^\infty\frac{G\lz \chi_{ml},q_1q^2_2\pz\psi(m)\psi'(q^2_2) }{m^wq^{2s}_{2}} \; \mbox{and} \;
  D_2(s, w; l, q_1, \psi,\psi') := \sum_{\substack{m,q_2=1 \\ (m,2)=1 \\ ml=\square}}^\infty\frac{G\lz \chi_{ml},q_1q^2_2\pz\psi(lm)\psi'(q^2_2) }{m^{w}q^{2s}_{2}}.
\end{split}
\end{align*}

   Plainly, the functions $D_i$ are defined for $\Re(s), \Re(w)$ large enough by Lemma \ref{lem:Gauss}. We now have the following result concerning the analytic properties of $D_i(s, w;l, q_1, \psi, \psi')$, which is a generalization of \cite[Lemma 3.4]{G&Zhao20}.
\begin{lemma}
\label{Estimate For D(w,t)}
 Under the same notations as above and with $\psi=\psi_1$ or $\psi_2$, the functions $D_i(s, w; l, q_1, \psi,\psi')$, with $i=1$, $2$ have meromorphic continuations to the region
\begin{align}
\label{Dregion}
		\{(s,w): \Re(s)>1, \ \Re(w)> 1 \}.
\end{align}

   If we assume the truth of RH, then the functions $D_i(s, w; l, q_1, \psi,\psi')$, with $i=1$, $2$ have meromorphic continuations to the region
\begin{align}
\label{DregionRH}
		\{(s,w): \Re(s)>1, \ \Re(w)> 3/4 \}.
\end{align}
    The only poles in the above regions occur for $D_1(s, w; q_1, \psi,\psi')$ at $w = 3/2$ when $q_1=1,  \psi=\psi_1$ and when $q_1=2,  \psi=\psi_2$ and the corresponding pole is simple in both cases.  For $\Re(s) \geq 1+\varepsilon$, $1+\varepsilon \leq \Re(w)\leq 2$, away from the possible
poles,
\begin{align}
\label{Diest}
			|D_i(s, w; l, q_1, \psi,\psi')|\ll l^{1/2+\varepsilon}(q_1(1+|w|))^{\max \{ (3/2-\Re(w))/2, 0 \}+\varepsilon}.
\end{align}	
   Moreover, in the region $\Re(s) \geq 1+\varepsilon$,
\begin{align}
\label{D1res}
\begin{split}
&\res_{w=3/2}D_1(s,w;l,1, \psi_1, \psi_1)=l^{1/2} \frac {\zeta^{(2)}(2s)}{2\zeta^{(2)}(2)} \prod_{p|l} \Big(1+\frac {1}{p} \Big)^{-1} .
\end{split}
\end{align}
   The relations \eqref{Diest} and \eqref{D1res} also hold, under RH, in the region $\Re(s) \geq 1+\varepsilon$, $3/4+\varepsilon \leq \Re(w)\leq 2$.
\end{lemma}
\begin{proof}
   We establish the lemma only for $D_1(s, w; l, q_1, \psi,\psi')$ here since the proof for $D_2(s, w; l, q_1, \psi,\psi')$ is similar. By virtue of Lemma~\ref{lem:Gauss} that $D_i, i=1,2$ admit Euler products
\begin{align*}
%%\label{D1Eulerprod}
\begin{split}	
 &	D_1(s, w; l, q_1, \psi,\psi')= \prod_p D_{1,p}(s, w; l, q_1,  \psi,\psi'),
\end{split}
\end{align*}
  where
\begin{align}
\label{Dexp}
\begin{split}	
 &	D_{1,p}(s, w; l;q_1,  \psi,\psi')= \displaystyle
\begin{cases}
\displaystyle \sum_{k=0}^\infty\frac{ \psi'(p^{2k})}{p^{2ks}}, &  \mbox{if} \; p|2, \\
\displaystyle \sum_{h,k=0}^\infty\frac{ \psi(p^{h+\text{ord}_p(l)})\psi'(p^{2k})G\lz \chi_{p^{h+\text{ord}_p(l)}}, q_1p^{2k} \pz  }{p^{hw+2ks}}, & \mbox{if} \; p \nmid 2.
\end{cases}
\end{split}
\end{align}

  The case $p \nmid l$ of $D_{1,p}$ has already been evaluated in \cite[(3.23), (3.24)]{G&Zhao20}. In particular, we have for $\Re(s)>1$, $\Re(w)>3/4$ and $p \nmid 2q_1l$,
\begin{align}
\label{Dgenexp}	
\begin{split}
    D_{1,p}(s, w; l, q_1,  \psi,\psi') =& 1+\frac{ \psi(p)\chi^{(q_1)}(p)}{p^{w-1/2}} +O \Big ( p^{-2\Re(s)}+p^{-2\Re(s)-\Re(w)+1}+p^{-2\Re(s)-4\Re(w)+4} \Big )  \\
   =& \frac {L_{p}\lz w-\tfrac{1}{2}, \chi^{(q_1)}\psi\pz}{\zeta_{p}(2w-1)}\lz 1+O \Big ( p^{-2\Re(s)}+p^{-2\Re(s)-\Re(w)+1}+p^{-2\Re(s)-4\Re(w)+4} \Big ) \pz .
\end{split}
\end{align}
Here for an $L$-function, $L_p$ denotes the factor in its Euler product for the prime $p$.  The above now allows us to deduce the meromorphic continuation in the regions \eqref{Dregion} and \eqref{DregionRH} by noticing here that the only difference between these two cases is due to the presence of poles of $\zeta^{-1}(2w-1)$ since it has no pole when $\Re(w)>1$ unconditionally and even in the larger region $\Re(w)>3/4$ under RH. \newline

Hence, we shall in what follows only establish the statements concerning \eqref{Diest} and \eqref{D1res} unconditionally. We also deduce from  \eqref{Dgenexp} that the only poles in the region given in \eqref{Dregion} or \eqref{DregionRH} are at $w = 3/2$ only when $q_1=1$, $\psi=\psi_1$ or when $q_1=2, \psi=\psi_2$. In either case, the corresponding pole is simple. \newline

  Next, as $l$ is square-free, we get from \eqref{Dexp} that if $p|l$,
\begin{align}
\label{Dexpl}
\begin{split}	
 	D_{1,p}(s, w; l, q_1,  \psi,\psi')=&
\displaystyle p^{w}\sum_{\substack{h \geq 1,k \geq 0}}\frac{ \psi(p^{h})\psi'(p^{2k})G\lz \chi_{p^{h}}, q_1p^{2k} \pz  }{p^{hw+2ks}} \\
=& p^{w}\Big (\sum_{\substack{h \geq 0,k \geq 0}}\frac{ \psi(p^{h})\psi'(p^{2k})G\lz \chi_{p^{h}}, q_1p^{2k} \pz  }{p^{hw+2ks}}-\sum_{k \geq 0}\frac{ G\lz \chi_{1}, p^{2k} \pz }{p^{2ks}}\Big ) \\
=& p^{w}\Big (\sum_{\substack{h \geq 0,k \geq 0}}\frac{ \psi(p^{h})\psi'(p^{2k})G\lz \chi_{p^{h}}, q_1p^{2k} \pz  }{p^{hw+2ks}}- \frac{1}{1-p^{-2s}} \Big ).
\end{split}
\end{align}
The last double sums above equal exactly $D_{1,p}$ by latter's definition in \eqref{Dexp} for $p \nmid 2l$. It follows from this and \eqref{Dgenexp} that for $\Re(s)>1$, $3/4<\Re(w)<2$ and $p|l$,
\begin{align*}
%%\label{Dgenexpl}	
\begin{split}
    D_{1,p}(s, w; l, q_1,  \psi,\psi') = &  p^{w}\Big (1+\frac{ \psi(p)\chi^{(q_1)}(p)}{p^{w-1/2}} +O \Big ( p^{-2\Re(s)}+p^{-2\Re(s)-\Re(w)+1}+p^{-2\Re(s)-4\Re(w)+4} \Big )  -(1-p^{-2s})^{-1}\Big ) \\
   =& p^{1/2}\Big (\psi(p)\chi^{(q_1)}(p)+O \Big ( p^{-2\Re(s)+\Re(w)-1/2}+p^{-2\Re(s)+1/2}+p^{-2\Re(s)-3\Re(w)+7/2} \Big )\Big ).
\end{split}
\end{align*}

    We use the above and argue as in the proof of \cite[Lemma 3.4]{G&Zhao20}.  This lead to to the conclusion that away from the possible
poles, the bound in \eqref{Diest} holds in the region $\Re(s) \geq 1+\varepsilon$, $3/4+\varepsilon \leq \Re(w)\leq 2$. \newline

  Lastly, we establish \eqref{D1res} by noticing that $D_{1,p}(s, w; l, 1,  \psi_1,\psi_1)$ has already been computed in the proof of \cite[Lemma 3.4]{G&Zhao20}, so that for $p \nmid 2l$,
\begin{align}
\label{D1pexp}	
\begin{split}
   D_{1,p}(s, w;  l, 1,  \psi_1,\psi_1)=& \zeta_p(w-\tfrac{1}{2})Q_p(s,w;l),
\end{split}
\end{align}
  where
\begin{align}
\label{Qpexp}	
\begin{split}
   Q_{p}(s, w;l)\Big |_{w=3/2} =& \Big(1-\frac {1}{p^{2}} \Big)(1-p^{-2s})^{-1}.
\end{split}
\end{align}

    We deduce from the above and \eqref{Dexpl} that for $p | l$,
\begin{align}
\label{D1pexpl}	
\begin{split}
   D_{1,p}(s, w;   l,  1,  \psi_1,\psi_1)=& p^{w-1}\Big(1+\frac {1}{p} \Big)^{-1}\zeta_p(w-\tfrac{1}{2})Q_p(s,w;l),
\end{split}
\end{align}
  where the restriction of $Q_{p}(s, w; l)$ at $w=3/2$ has the same value as that given in \eqref{Qpexp}. \newline

  We recall that the conductor of $\psi'$ divides $8$, so that for any $\psi'$ and any $p>2$,
\begin{align} 
\label{Dp2a}
\begin{split}	
\sum_{k=0}^\infty\frac{ \psi'(p^{2k})}{p^{2ks}}= (1-p^{-2s})^{-1}.
\end{split}
\end{align}
  The above equality continues to hold for $p=2, \psi'=\psi_0$ while the left-hand side expression \eqref{Dp2a} equals $0$ for $p=2$, $\psi' \neq \psi_0$. \newline
	
   These observations enable us to derive from \eqref{Dexp}, \eqref{D1pexp}--\eqref{D1pexpl} that
\begin{align*}
%%\label{D1q11}
\begin{split}	
 &	D_1(s, w; l, 1, \psi_1,\psi_1)= l^{w-1} \prod_{p|l} \Big(1+\frac {1}{p} \Big)^{-1} \zeta^{(2)}(w-\tfrac{1}{2})Q(s,w;l),
\end{split}
\end{align*}
  where $Q(s,w;l)=\prod_{p>2}Q_p(s,w;l)$ with $Q_p(s,w;l)$ defined in \eqref{D1pexp}, \eqref{D1pexpl}.  Mindful that the residue of $\zeta(s)$ at $s=1$ is $1$, we readily deduce \eqref{D1res}, completing the proof of the lemma.
\end{proof}

  We now apply Lemma \ref{Estimate For D(w,t)} with \eqref{Cidef} to see that $(w-3/2)C_i(s,w; l, \psi,\psi')$, $i=1$, $2$ are defined in the region
\begin{equation*}
%%\label{key}
\begin{cases}
\displaystyle \{(s,w):\ \Re(s)>1, \ 1<\Re(w)<2, \ \Re(s+w/2)>7/4 \}, & \quad \text{unconditionally}, \\
\displaystyle \{(s,w):\ \Re(s)>1, \ 3/4<\Re(w)<2, \ \Re(s+w/2)>7/4 \}, & \quad \text{under RH}.
\end{cases}		
\end{equation*}
	The above together with \eqref{A1A2}, \eqref{S3} and \eqref{A2exp} now implies that $(s-1)(w-1)(s+w-3/2)A(s,w;l)$ can be extended to the region
\begin{equation*}
%%\label{key}
\begin{cases}
		S_4=\{(s,w):\ 1<\Re(s+w)<2,\ \Re(w-s)>3/2, \ \Re(s)<0\}, & \quad \text{unconditionally},\\
    S'_4=\{(s,w):\ 3/4<\Re(s+w)<2,\ \Re(w-s)>3/2, \ \Re(s)<0\}, & \quad \text{under RH}.
\end{cases}	
\end{equation*}
  Now, the convex hull $S_5$ (resp. $S'_5$) of $S_2 \cup S_4$ (resp. $S_2 \cup S'_4$) is
\begin{align*}
%%\label{key}
		S_5= \{(s,w):\ \Re(s+w)>1 \}, \quad \mbox{and} \quad
   S'_5= \{(s,w):\ \Re(s+w)>3/4 \}.
\end{align*}
Now Theorem \ref{Bochner} implies that $(s-1)(w-1)(s+w-3/2)A(s,w;l)$ admits analytic continuation to $S_5$ unconditionally and to $S'_5$ under RH.
	
\subsection{Residues}
\label{sec:resA}
	
	We see from \eqref{Sum A(s,w,z) over n} that $A(s,w;l)$ has a pole at $s=1$ arising from the terms with $ml=\square$. As the residue of $\zeta(s)$ at $s = 1$ equals $1$, we deduce that
\begin{align}
\label{Residue at s=1}
 \res_{s=1}& A(s, \tfrac{1}{2}+\alpha; l) = \res_{s=1} A_1(s, \tfrac{1}{2}+\alpha; l)
=\frac {\zeta^{(2)}(1+2\alpha)}{2l^{1/2+\alpha}\zeta^{(2)}(2+2\alpha)}\prod_{p|l}\frac{1-p^{-1}}{1-p^{-2-2\alpha}} .
\end{align}

    Next, we keep the notation from the proof of Lemma \ref{Estimate For D(w,t)}. We deduce from \eqref{A2exp} and Lemma \ref{Estimate For D(w,t)} that $A(s,w;l)$ has a pole at $s+w=3/2$ and
\begin{align}
\label{Cres}
\begin{split}
\res_{s=3/2-w}A(s,w;l)=& \res_{s=3/2-w}A_2(s,w;l) \\
=& -\frac{2^s\pi^{s-1/2}}{(4l)^s}\frac {\Gamma_{e}(s)+\Gamma_{o}(s)}{2}\res_{s=3/2-w}C_1(1-s,w+s;l,\psi_1, \psi_1)\\
&\hspace*{1cm} +\frac{2^{2s-1}\pi^{s-1/2}}{(4l)^s}\frac {\Gamma_{e}(s)+\Gamma_{o}(s)}{2}\res_{s=3/2-w} \psi_2(l)C_1(1-s,w+s;l,\psi_2, \psi_0).
\end{split}
\end{align}

Using \eqref{C12def} to write $C_1(1-s,w+s;l,\psi_2, \psi_0)$ as a double series and applying Lemma \ref{Estimate For D(w,t)}, we see that the pole at
$s+w=3/2$ occurs only when $2 |q$. Upon making a change of variable $q\rightarrow 2q$, we deduce from Lemma \ref{Estimate For D(w,t)} again that
\begin{align}
\label{Cres1}
\begin{split}
 & \frac{2^{2s-1}\pi^{s-1/2}}{(4l)^s}\frac {\Gamma_{e}(s)+\Gamma_{o}(s)}{2}\res_{s=3/2-w} \psi_2(l)C_1(1-s,w+s;l,\psi_2, \psi_0) \\
=& \frac{2^{3s-2}\pi^{s-1/2}}{(4l)^s}\frac {\Gamma_{e}(s)+\Gamma_{o}(s)}{2}\res_{s=3/2-w} C_1(1-s,w+s;l,\psi_1, \psi_0).
\end{split}
\end{align}

By \eqref{Cidef}, we see that for $i=0,1$, 
\begin{align}
\label{Cres2}
\begin{split}
  \res_{s=3/2-w} C_1(1-s,w+s; l,\psi_1, \psi_i)= \res_{s=3/2-w} D_1(1-s,w+s;l,1, \psi_1, \psi_i).
\end{split}
\end{align}

   Note also from \eqref{Dexp} and \eqref{Dp2a} that 
\begin{align}
\label{Cres3}
\begin{split}
  D_1(1-s,w+s;l,1, \psi_1, \psi_0)=(1-2^{-2(1-s)})^{-1}D_1(1-s,w+s;l,1, \psi_1, \psi_1).
\end{split}
\end{align}
  It follows from \eqref{Cres1}--\eqref{Cres3} that
\begin{align*}
%%\label{Cres1psi1}
\begin{split}
 & \frac{2^{3s-2}\pi^{s-1/2}}{(4l)^s}\frac {\Gamma_{e}(s)+\Gamma_{o}(s)}{2}\res_{s=3/2-w} C_1(1-s,w+s;l,\psi_1, \psi_0) \\
=& \frac{2^{3s-2}\pi^{s-1/2}}{(4l)^s}(1-2^{-2(1-s)})^{-1}\frac {\Gamma_{e}(s)+\Gamma_{o}(s)}{2}\res_{s=3/2-w} D_1(1-s,w+s;l, 1, \psi_1, \psi_1).
\end{split}
\end{align*}

Substituting the above into \eqref{Cres} and making use of \eqref{Cres2} for the case $i=0$ renders that
\begin{align}
\label{Ares}
\begin{split}
& \res_{s=3/2-w}A(s,w;l) \\
=& (2^{3s-2}(1-2^{-2(1-s)})^{-1}-2^s)\frac{\pi^{s-1/2}}{(4l)^s}\frac {\Gamma_{e}(s)+\Gamma_{o}(s)}{2}\res_{s=3/2-w}D_1(1-s,w+s;l,q,\psi_1, \psi_1) \\
=& 2^{s-2}\frac{1-2^{-(1-2(1-s))}}{1-2^{-2(1-s)}} \frac{\pi^{s-1/2}}{l^s}(\Gamma_{e}(s)+\Gamma_{o}(s))\res_{s=3/2-w}D_1(1-s,w+s;l,q, \psi_1, \psi_1).
\end{split}
\end{align}

   We then set $w=1/2+\alpha$ so that $s=1-\alpha$ to see from the above, \eqref{D1res} and \eqref{Ares} that
\begin{align}
\label{Areswfixed}
\begin{split}
& \res_{s=1-\alpha}A(s,1/2+\alpha;l) = 2^{-1-\alpha}(1-2^{-(1-2\alpha)})\pi^{1/2-\alpha}l^{-1/2+\alpha}(\Gamma_{e}(1-\alpha)+\Gamma_{o}(1-\alpha))\frac {\zeta(2\alpha)}{2\zeta^{(2)}(2)} \prod_{p|l} \Big(1+\frac {1}{p} \Big)^{-1} .
\end{split}
\end{align}

Next, by \cite[Lemma 2.5]{Cech1}, we have
\begin{equation*}
%%\label{gammasum}
			\Go(s)+\Ge(s)=\frac{2^{s+1/2}\Gamma(1-s)\cos\lz \frac{\pi }{2}(s-\frac 12)\pz}{\sqrt \pi}.
\end{equation*}
   The above with $s=1-\alpha$ leads to
\begin{equation*}
%%\label{gammasum1}
			\Go(1-\alpha)+\Ge(1-\alpha)=\frac{2^{3/2-\alpha}\Gamma(\alpha)\cos\lz \frac{\pi}{2} (\frac 12-\alpha)\pz}{\sqrt \pi}.
\end{equation*}
	
  Note that the functional equation \eqref{fcneqnprimitive} for $n=1$ and $\chi$ being the principal character modulo $1$ implies that
\begin{align*}
%%\label{zetafcneqn}
  \zeta(2\alpha)=\pi^{2\alpha-1/2}\frac {\Gamma(\tfrac{1}{2}-\alpha)}{\Gamma (\alpha)}\zeta(1-2\alpha).
\end{align*}

Therefore,
\begin{align*}
%%\label{zetafcneqn}
  \lz \Go(1-\alpha)+\Ge(1-\alpha) \pz \zeta(2\alpha)=\pi^{2\alpha-1}2^{3/2-\alpha}\Gamma(\tfrac{1}{2}-\alpha)\cos\lz \tfrac{\pi}{2} (\tfrac 12-\alpha)\pz\zeta(1-2\alpha).
\end{align*}

    We further note the following relation (see \cite[\S 10]{Da})
\begin{align*}
%%\label{zetafcneqn}
 \frac {\Gamma (\frac s2)}{\Gamma (\frac 12-\frac s2)}=\pi^{-1/2}2^{1-s}\cos(\frac {\pi}{2}s)\Gamma(s).
\end{align*}

   We then deduce that
\begin{align}
\label{gammazeta}
  \lz \Go(1-\alpha)+\Ge(1-\alpha) \pz \zeta(2\alpha)=\pi^{2\alpha-1/2}2^{1-2\alpha}\frac {\Gamma(\tfrac{1}{4}-\tfrac{\alpha}{2})}{\Gamma(\tfrac{1}{4}+\tfrac{\alpha}{2})}.
\end{align}

From \eqref{Areswfixed} and \eqref{gammazeta},
\begin{align}
\label{Areswexplicit}
\begin{split}
& \res_{s=1-\alpha}A(s,1/2+\alpha; l) = \gamma_{\alpha}l^{-1/2+\alpha} \frac {\zeta^{(2)}(1-2\alpha)}{2\zeta^{(2)}(2)} \prod_{p|l} \Big(1+\frac {1}{p} \Big)^{-1} ,
\end{split}
\end{align}
  where $\gamma_{\alpha}$ is defined in \eqref{gammadef}.

\subsection{Completion of the proof}

The Mellin inversion gives that
\begin{equation}
\label{Integral for all characters}
		\sum_{(d,2)=1}L(\tfrac 12, \chi^{(8d)})\chi^{(8d)}(l)w \bfrac {d}X=\frac1{2\pi i}\int\limits_{(2)}A\lz s,\tfrac12+\alpha;l\pz X^s\widehat w(s) \dif s,
\end{equation}
  where $A(s, w; l)$ defined in \eqref{Aswfexp} and $\widehat{w}$ is the Mellin transform of $w$ given by
\begin{align*}
     \widehat{w}(s) =\int\limits^{\infty}_0w(t)t^s\frac {\dif t}{t}.
\end{align*}
   Observe that integration by parts implies that for any integer $E \geq 0$,
\begin{align}
\label{whatbound}
 \widehat w(s)  \ll  \frac{1}{(1+|s|)^{E}}.
\end{align}

  We define
\begin{equation*}
%%\label{Definition of S tilde}
		\widetilde S_5=S_{5,\delta}\cap\{(s,w):\Re(s) \geq -5/2 \} \quad \mbox{and} \quad \widetilde S'_5=S'_{5,\delta}\cap\{(s,w):\Re(s) \geq -5/2 \},
\end{equation*}
	where $\delta$ is a fixed number with $0<\delta <1/1000$, $S_{5,\delta}= \{ (s,w)+\delta (1,1) : (s,w) \in S_5 \}$ and $S'_{5,\delta}= \{ (s,w)+\delta (1,1) : (s,w) \in S'_5 \}$. We further set
\begin{equation*}
%%\label{key}
		p(s,w)=(s-1)(w-1)(s+w-3/2), \quad \tilde p(s,w)=1+|p(s,w)|.
\end{equation*}
  We argue as in \cite[Section 3.6]{G&Zhao20}, upon using Proposition \ref{Extending inequalities} and Lemma \ref{Estimate For D(w,t)} to see that in $\widetilde S_5$ or $\widetilde S'_5$,
\begin{equation}
\label{AboundS4}
		|p(s,w)A(s,w;l)|\ll l^{1/2+\varepsilon} \tilde p(s,w)(1+|w|)^{2+\varepsilon}(1+|s|)^{5+\varepsilon}.
\end{equation}

  Given \eqref{whatbound} and \eqref{AboundS4}, we now shift the line of integration in \eqref{Integral for all characters} to $\Re(s)=1/2-\Re(\alpha)+\varepsilon$ unconditionally or $\Re(s)=1/4-\Re(\alpha)+\varepsilon$ under RH.  The contribution of the integral on the new line is bounded by the $O$-term in \eqref{Asymptotic for ratios of all characters}.  In this process, we also encounter two simple poles at $s=1$ and $s=1-\alpha$ such that the corresponding residues are given in \eqref{Residue at s=1} and \eqref{Areswexplicit}, respectively. Direct computations now lead to the main terms given in \eqref{Asymptotic for ratios of all characters}. This completes the proof of Theorem \ref{Theorem for all characters}.

\section{Proof of Theorem \ref{theo:recursive}}

If $(r,l)=1$, then product $rl$ remains square-free if both $r$ and $l$ are. It follows from this that we may apply our induction hypothesis that \eqref{Asymptotic for ratios of all characters} holds with an error term of size $l^{1/2 + \varepsilon} X^{\delta+ \varepsilon}$ for any $0< \delta \leq 1/2$ and any $\Re(\alpha) \ll (\log{X})^{-1}$, $\Im(\alpha) \ll X^{\varepsilon}$.  Hence for such $\alpha$ and any odd, square-free $rl>0$,
\begin{align*}
%%\label{Asymptotic for ratios of all charactersrl}
\begin{split}	
\sum_{\substack{(d,2)=1}} & L(\tfrac{1}{2}+\alpha, \chi^{(8d)})\chi^{(8d)}(rl) w \bfrac {da^2}X \\
=&  \frac{X}{a^2} \frac {\M w(1) \zeta^{(2)}(1+2\alpha)}{2(rl)^{1/2+\alpha}\zeta^{(2)}(2+2\alpha)}\prod_{p|rl} \frac{1-p^{-1}}{1-p^{-2-2\alpha}} \\
&\hspace*{1cm} +\leg{X}{a^2}^{1-\alpha}\M w(1-\alpha) \gamma_{\alpha}(rl)^{-1/2+\alpha} \frac {\zeta^{(2)}(1-2\alpha)}{2\zeta^{(2)}(2)}\prod_{p|rl} \Big(1+p^{-1} \Big)^{-1} +O\Big((lr)^{1/2 + \varepsilon}  \leg{X}{a^2}^{\delta + \varepsilon}\Big ).
\end{split}
\end{align*}
  We substitute the above into the right-hand side of \eqref{M1simplified} to get $M_1=M_{11}+M_{12}+E_1$, where
\begin{align}
\label{MME}
\begin{split}
 M_{11}=& \sum_{\substack{a \leq Y \\ (a, 2l)=1}}\mu(a)\sum_{r | a}\frac {\mu(r)}{r^{1+2\alpha}}\frac{X}{a^2} \M w(1)\frac {\zeta^{(2)}(1+2\alpha)}{2l^{1/2+\alpha}\zeta^{(2)}(2+2\alpha)}\prod_{p|rl} \frac{1-p^{-1}}{1-p^{-2-2\alpha}} , \\
 M_{12}=& \sum_{\substack{a \leq Y \\ (a, 2l)=1}}\mu(a)\sum_{r | a}\frac {\mu(r)}{r}\leg{X}{a^2}^{1-\alpha}\M w(1-\alpha) \gamma_{\alpha}l^{-1/2+\alpha} \frac {\zeta^{(2)}(1-2\alpha)}{2\zeta^{(2)}(2)} \prod_{p|rl} \Big(1+p^{-1} \Big)^{-1}  +O\Big((lr)^{1/2 + \varepsilon}  \leg{X}{a^2}^{\delta + \varepsilon}\Big ) \\
 E_1 \ll & \sum_{\substack{a \leq Y \\ (a, 2l)=1}} \sum_{r | a} \frac{1}{r^{1/2}} (lr)^{1/2 + \varepsilon} \leg{X}{a^2}^{\delta + \varepsilon} \ll X^{\delta + \varepsilon}Y^{1-2\delta} l^{1/2 + \varepsilon}.
\end{split}
\end{align}

Simplifying the above, we arrive at
\begin{align} \label{M1maintermss1alpha1}
  M_{11}=\frac {X\widehat w(1)}{2l^{1/2+\alpha}} \sum_{\substack{a \leq Y \\ (a, 2l)=1}}\frac {\mu(a)}{a^2}\frac {\zeta^{(2a)}(1+2\alpha)}{\zeta^{(2a)}(2+2\alpha)}\prod_{p|l}\frac{1-p^{-1}}{1-p^{-2-2\alpha}} ,
\end{align}
and
\begin{align}
\label{M1maintermss1alpha2}
 M_{12}= \frac {X^{1-\alpha}\widehat w(1-\alpha)}{2\zeta^{(2)}(2)}l^{-1/2+\alpha}\gamma_{\alpha} \Big (\prod_{p|l} \Big(1+\frac {1}{p} \Big)^{-1} \Big )\sum_{\substack{a \leq Y \\ (a, 2l)=1}}\frac {\mu(a)}{a^{2-2\alpha}}\zeta^{(2)}(1-2\alpha)\prod_{p|a} \Big(1+\frac {1}{p} \Big)^{-1} +O\Big((lr)^{1/2 + \varepsilon}  \leg{X}{a^2}^{\delta + \varepsilon}\Big ) .
\end{align}

  Similarly, we may apply our induction hypothesis that \eqref{Asymptotic for ratios of primitive characters} holds with an error term of size $\ll l^{1/2 + \varepsilon} X^{f+ \varepsilon}$ for any $1/2 \leq f \leq 1$ and any $\Re(\alpha) \ll (\log{X})^{-1}$, $\Im(\alpha) \ll X^{\varepsilon}$.  Consequently, for such $\alpha$ and any odd, square-free $rl>0$,
\begin{align}
%%\label{Areswexplicit}
\begin{split}
 \sumstar_{(d,2)=1} &L(\tfrac{1}{2}+\alpha, \chi^{(8d)})\chi^{(8d)}(lr)w \bfrac {dc^2}X \\
=&  \frac{X}{c^2} \frac{\widehat{w}(1)}{2 \zeta^{(2)}(2)} (lr)^{-1/2-\alpha} \zeta^{(2)}(1 + 2\alpha) B_{\alpha}(lr)+
\leg{X}{c^2}^{1-\alpha} \frac{\widehat{w}(1-\alpha)\gamma_{\alpha}}{2 \zeta^{(2)}(2)}  (lr)^{-1/2+\alpha} \zeta^{(2)}(1 - 2\alpha) B_{-\alpha}(lr) \\
& \hspace*{1cm} + O\Big((lr)^{1/2 + \varepsilon}  \leg{X}{c^2}^{f + \varepsilon}\Big ).
\end{split}
\end{align}

So from \eqref{M2simplifiedc}, $M_2=M_{21}+M_{22}+E_2$, where
\begin{align*}
%%\label{MME2}
\begin{split}
 M_{21}=& \frac{X \widehat w(1) }{2l^{1/2+\alpha}\zeta^{(2)}(2)} \sum_{(c,2l)=1}\frac{1}{c^2} \Big(\sum_{\substack{a|c \\ a > Y }}\mu(a)\Big)\sum_{r | c}\frac {\mu(r)}{r^{1+2\alpha}} \zeta^{(2)}(1 + 2\alpha)B_{\alpha}(lr)  := \frac{X \widehat w(1) }{2l^{1/2+\alpha}\zeta^{(2)}(2)}C(\alpha, \alpha), \\
 M_{22}=& \frac {X^{1-\alpha}\widehat w(1-\alpha)}{2\zeta^{(2)}(2)}l^{-1/2+\alpha}\gamma_{\alpha} \sum_{(c,2l)=1}\frac{1}{c^{2-2\alpha}} \Big(\sum_{\substack{a|c \\ a > Y }}\mu(a)\Big)\sum_{r | c}\frac {\mu(r)}{r} \zeta^{(2)}(1 - 2\alpha)B_{-\alpha}(lr) \\
 :=&  \frac {X^{1-\alpha}\widehat w(1-\alpha)}{2\zeta^{(2)}(2)}l^{-1/2+\alpha}\gamma_{\alpha} C(-\alpha, \alpha), \\
 E_2 \ll & X^{\varepsilon} \sum_{(c,2l)=1} \lp \sum_{\substack{a |c \\ a > Y}} 1 \rp  \sum_{r | c} \frac{1}{r^{1/2}} (lr)^{1/2 + \varepsilon} \leg{X}{c^2}^{f + \varepsilon} \ll \frac{X^{f + \varepsilon}}{Y^{2f - 1}} l^{1/2 + \varepsilon}.
\end{split}
\end{align*}
and $C(\pm \alpha, \alpha)$ are defined in the first display on \cite[p.83]{Young1}.  The expression for $C(\alpha, \alpha)$ given in the display above \cite[(4.3)]{Young1} gives that
\begin{align}
\label{M2maintermss12}
 M_{21}=\frac{X \widehat w(1) }{2l^{1/2+\alpha}}\sum_{\substack{a > Y \\ (a, 2l)=1}}\frac {\mu(a)}{a^2}\frac {\zeta^{(2a)}(1+2\alpha)}{\zeta^{(2a)}(2+2\alpha)}\prod_{p|l}\frac{1-p^{-1}}{1-p^{-2-2\alpha}} .
\end{align}
The simplification for $C(-\alpha, \alpha)$ given on \cite[p.85]{Young1} leads to
\begin{align}
\label{M2maintermss12alpha}
  M_{22}=\frac {X^{1-\alpha}\widehat w(1-\alpha)}{2\zeta^{(2)}(2)}l^{-1/2+\alpha}\gamma_{\alpha} \sum_{\substack{a > Y \\ (a, 2l)=1}}\frac {\mu(a)}{a^{2-2\alpha}}\zeta^{(2)}(1-2\alpha)\prod_{p|a} \Big(1+\frac {1}{p} \Big)^{-1}.
\end{align}

From \eqref{M1maintermss1alpha1} and \eqref{M2maintermss12},
\begin{align}
\label{maintermss1}
\begin{split}
 M_{11}+M_{21}=&\frac{X \widehat w(1) }{2l^{1/2+\alpha}}\sum_{\substack{(a, 2l)=1}}\frac {\mu(a)}{a^2}\frac {\zeta^{(2a)}(1+2\alpha)}{\zeta^{(2a)}(2+2\alpha)}\prod_{p|l}\frac{1-p^{-1}}{1-p^{-2-2\alpha}}  = \frac{X \widehat{w}(1)}{2 \zeta^{(2)}(2)} l^{-1/2 - \alpha} \zeta^{(2)}(1 + 2\alpha) B_{\alpha}(l),
\end{split}
\end{align}
 where the last equality above follows from the computations given in the last two displays on \cite[p.97]{Young1}. \newline

  We also deduce from \eqref{M1maintermss1alpha2} and \eqref{M2maintermss12alpha} that
\begin{align}
\label{maintermss2}
\begin{split}
  M_{12}+M_{22}=&\frac {X^{1-\alpha}\widehat w(1-\alpha)}{2\zeta^{(2)}(2)}l^{-1/2+\alpha}\gamma_{\alpha}\zeta^{(2)}(1-2\alpha)\sum_{\substack{(a, 2l)=1}}\frac {\mu(a)}{a^{2-2\alpha}}\prod_{p|a} \Big(1+\frac {1}{p} \Big)^{-1} \\
  =& \frac {X^{1-\alpha}\widehat w(1-\alpha)}{2\zeta^{(2)}(2)}l^{-1/2+\alpha}\gamma_{\alpha}\zeta^{(2)}(1-2\alpha) B_{-\alpha}(l).
\end{split}
\end{align}
Here the last equality above follows from \cite[(3.2)]{Young1}. \newline

Optimizing $Y$ by setting $Y=X^{1/2}$ yields
\begin{align}
\label{E}
\begin{split}
E_1+E_2 \ll l^{1/2 + \varepsilon}X^{1/2 + \varepsilon}.
\end{split}
\end{align}
Now Theorem \ref{theo:recursive} follows from \eqref{maintermss1}--\eqref{E}. This completes the proof.

\vspace*{.5cm}

\noindent{\bf Acknowledgments.}  P. G. is supported in part by NSFC grant 11871082 and L. Z. by the FRG Grant PS43707 at the University of New South Wales.  The authors are very grateful to the anonymous referee for his/her careful inspection of the paper and many helpful comments.

\bibliography{biblio}
\bibliographystyle{amsxport}

\end{document}